\documentclass[11pt]{amsart}
\usepackage{tikz-cd}
\usepackage{amsmath,amssymb, xypic,verbatim, amscd,color}
\usepackage[margin=0.9in]{geometry}
\usepackage{graphicx}
\usepackage{colordvi}
\usepackage{mathdots}
\usepackage[colorlinks=true, pdfstartview=FitV,
linkcolor=cyan, citecolor=red, urlcolor=blue]{hyperref}
\usepackage{amsfonts,latexsym}

\usepackage{thmtools}
\usepackage{thm-restate}
\usepackage{hyperref}
\usepackage{cleveref}
\numberwithin{equation}{section}

\usepackage{amsfonts,amsmath,url, amsthm,xypic,amssymb,graphicx,color,enumerate,mathrsfs, amscd}
\usepackage{bm}

\usepackage{enumerate,bbm}
\usepackage[english]{babel}

\usepackage{ulem}
\usepackage[all]{xy}

\usepackage{thmtools}
\usepackage{thm-restate}

\usepackage{hyperref}

\def\Gm{\mathbb{G}_m}

\def\Q{\mathbb{Q}}
\def\P{\mathbb{P}}
\def\C{\mathbb{C}}
\def\A{\mathbb{A}}
\def\Z{\mathbb{Z}}

\def\x{\times}
\def\ox{\otimes}
\def\a{\alpha}

\def\tlt0{(\tilde LT/T)_0}

\def\cO{\mathcal{O}}

\DeclareMathOperator{\Aff}{Aff}
\DeclareMathOperator{\ec}{Spec }
\DeclareMathOperator{\im}{im}
\DeclareMathOperator{\bun}{Bun}

\newcommand{\abcd}[4]{\left(\begin{array}{cc}
  #1 & #2 \\
  #3 & #4 \\
\end{array} \right)}

\newcommand{\ab}[2]{\left(\begin{array}{c}
  #1 \\
  #2 \\
\end{array} \right)}

\newcommand{\mc}[1]{\mathcal{#1}}

\newtheorem{thm}{Theorem}[section]
\newtheorem{prop}[thm]{Proposition}
\newtheorem{lemma}[thm]{Lemma}
\newtheorem{cor}[thm]{Corollary}
\newtheorem{conj}{Conjecture}

\theoremstyle{definition}

\theoremstyle{remark}
\newtheorem{rmk}{Remark}
\theoremstyle{notation}

\newtheorem{example}{Example}

\DeclareMathOperator{\codim}{codim}
\DeclareMathOperator{\rank}{rank}
\DeclareMathOperator{\proj}{Proj\;}

\usepackage{pgf}
\usepackage{tikz}
\usetikzlibrary{arrows,automata} 
\begin{document}

\title[Hunting Bundles]{
  Hunting Vector Bundles on $\P^1 \x \P^1$
} 
 
\author{ Pablo Solis}
\address{Department of Mathematics,
 Stanford, CA}
\email{pablos.inbox@gmail.com}

\thanks{I would like to thank Adam Boocher who brought this problem to my attention. This paper also benefitted from helpful conversations with Nathan Ilten, Claudiu Raicu, Tom Graber, and Xinwen Zhu. 
}

\begin{abstract}
Boij-S\"oderberg theory concerns resolutions of graded modules over a polynomial ring over a field. Specifically Boij-S\"oderberg theory gives a description of the cone of Betti diagrams for Cohen-Macaulay modules. Eisenbud and Schreyer discovered a duality between the cone of Betti diagrams and the cone of cohomology tables for vector bundles on projective space. In the dual theory an important role is played by so called natural vector bundles $E$ which have the property that the cohomology of every twist of $E$ is concentrated in a single degree. In \cite{EisenbudMR2810424}, Eisenbud and Schreyer consider the bi-graded theory on $\P^1 \x \P^1$ and conjectured that natural vector bundles exist with prescribed Hilbert polynomial. The Hilbert polynomial depends on three rational number $\a,\beta, \gamma$. We prove this conjecture provided that $\a,\beta$ are not both integral.
\end{abstract}


\maketitle

\tableofcontents

\section{Introduction}
The aim of this paper is to describe some vector bundles on $Q = \P^1 \x \P^1$ and compute the cohomology of these bundles and their twists. Specifically we seek bundles $E$ such that for all twists $E(n,m) = E\ox \mc{O}_{Q}(n,m)$ the cohomology $H^*(E(n,m))$ is concentrated in one degree; such bundles are said to have natural cohomology.

The motivation to search for these vector bundles comes from Boij-S\"oderberg theory and specifically a conjecture stated in \cite{EisenbudMR2810424} which says:
\begin{conj}[Eisenbud and Schreyer]\label{conj}
For any $p(x,y) = (x+\alpha)(y+\beta)-\gamma \in \Q[x,y]$ with $\gamma>0$ there exists a vector bundle $E$ with natural cohomology and Hilbert polynomial $\chi(E(a,b)) = \rank(E) p(a,b)$ for $\rank(E)$ sufficiently big.
\end{conj}

There are direct sums of line bundles on $Q$ which have natural cohomology and have $\gamma<0$ (e.g. $\cO \oplus \cO(1,1)$). But in \cite[Lemma 5.1]{EisenbudMR2810424} it is shown that if $E$ is not a direct sum of line bundles then the condition on $\gamma$ is necessary.

\newtheorem*{thm:test}{Theorem \ref{thm:main}}
\begin{thm:test}
Let $p(x,y) = (x+\alpha)(y+\beta)-\gamma \in \Q[x,y]$ with $\gamma>0$ and $\alpha, \beta$ not both integral. If $r>1$ is an integer such that $r p(x,y) \in \Z[x,y]$ then there is a vector bundle $E$ of rank $r$ with natural cohomology such that $\chi(E(a,b)) = r p(a,b)$.
\end{thm:test}

After applying a suitable integral shift $(x,y) \mapsto (x+n,y+m)$ it is straightforward to see the cases of conjecture \ref{conj} not covered by theorem \ref{thm:main} are when $p(x,y) = x y - \gamma$. However it seems plausible that an approach not too different than the one given here could be made to work for the remaining cases.

\subsection{Boij-S\"oderberg theory}
Boij-S\"oderberg theory originally concerned the cone of Betti table of Cohen-Macaulay modules $M$ over a polynomial ring over a field $R = k[x_1,\dotsc, x_n]$ \cite{MR2427053}. The Betti numbers $\beta_{i,j}(M)$ encode the minimal free resolution of $M$ and are a refinement of the Hilbert function of $M$: $\dim M_d = \sum_i (-1)^i(\sum_j \beta_{i,j}(M))$; two modules can have the same Hilbert function but different Betti numbers.

The numbers $\beta_{i,j}(M)$ can be considered as an integral vector $\beta_{\bullet,\bullet}(M)$ in an infinite but countable dimensional $\Q$ vector space $V$; $\beta_{\bullet,\bullet}(M)$ is called the Betti table of $M$. Essentially Boij-S\"oderberg theory studies Betti tables up to scalar multiple. Specifically one considers the cone $C_{BS}$ generated by Betti tables of Cohen-Macaulay modules. An introduction to this material can be found here \cite{ MR2932580}.

The geometric side of Boij-S\"oderberg theory concerns the cohomology of vector bundles $E$ on projective space $\P^n$. In this case one considers the numbers $\gamma_{i,j}(E) = \dim H^i(\P^n,E(j))$. Then the collection of all $\gamma_{i,j}(E)$ gives an integral vector $\gamma_{\bullet,\bullet}(E)$ in an infinite dimensional $\Q$ vector space and $\gamma_{\bullet,\bullet}(E)$ is called the cohomology table of $E$. We set $C_{ES}$ to be the cone of cohomology tables of vector bundles; this geometric side was formulated by Eisenbud and Schreyer in \cite{MR2505303} where it is shown that $C_{BS}$ and $C_{ES}$ are dual under a natural pairing. 

In fact cohomology tables were considered in \cite{MR1990756} even before Boij-S\"oderberg theory was formulated. This is not surprising because many interesting properties of a vector bundle are encoded in its cohomology table. For example $E$ on $\P^n$ is a direct sum of line bundles if and only if $H^i(\P^n,E) = 0$ for $0<i<n$; this is due to Horrocks and is a special case of \cite[7.5,9.4]{MR0169877} although the statement and proof is more easily extracted from \cite[sect.\;5]{MR509589}.

Of course for any point $v$ in any polyhedral cone $C$ we can express $v$ as a finite sum of vectors on the extremal rays of $C$. For the cones $C_{BS}$ and $C_{ES}$ this statement becomes more significant when coupled with the fact that one can always find an integral point on the extremal rays that is represented by a module/vector bundle. Betti tables which are extremal are called pure resolutions and cohomology tables which are extremal are called supernatural bundles. Thus for example the cohomology table of an arbitrary vector bundle $E$ can be expressed as a $\Q$-linear combination of cohomology tables of supernatural bundles $E_i^{super}$:
\begin{equation}\label{eq:super}
\gamma_{\bullet,\bullet}(E) = \sum_i c_i \gamma_{\bullet,\bullet}(E_i^{super}), \ \ c_i \in \Q.
\end{equation}
It is natural ask if the numerical relationship in equation \ref{eq:super} can be ``lifted'' to a geometric relationship between $E$ and $E_i^{super}$. In fact Erman and Sam show this is possible in some cases \cite{MR3594287}.
  
A good amount is known about the cone of cohomology table $C_{ES}(X)$ on other projective varieties $X$. For example if $X \subset \P^n$ is $d$ dimensional, $\pi \colon X \to \P^d$ is a finite linear projection and if $F \in Coh(X)$ has the property that $\pi_*F = \cO_{\P^d}^r$ for some $r>0$ then $C_{ES}(X) = C_{ES}(\P^d)$. Such a sheaf $F$ is called an Ulrich sheaf. Determining which $X \subset \P^n$ have Ulrich sheaves is an open problem. Perhaps they all do? This is a question asked in \cite{MR1969204}.

In contrast much less is known in a multigraded setting. There is a duality theorem \cite{MR3692884}, which applies to toric variety that is defined at the level of derived categories. There are also some results about the cone of bigraded Betti tables \cite{MR3051372,MR2810421,MR2727620}. But in some sense conjecture \ref{conj} captures how little is known about cohomology tables of vector bundles on $\P^1 \x \P^1$. 

In general if $X$ is a projective variety and $L_1, \dotsc, L_n \in Pic(X)$ are line bundles then a vector bundle $E$ on $X$ has {\it natural cohomology} if for every $(k_1,\dotsc, k_n$ the cohomology of the twist $E\ox L_1^{k_1} \ox \dotsb \ox L_n^{k_n}$ is concetrated in one degree. The most relevant cases are $\Z$ graded setting of $\cO_{\P^n}(1)$ on $\P^n$ and the bi-graded setting of $\cO_{\P^1 \x \P^1}(1,0),\cO_{\P^1 \x \P^1}(0,1)$ on $\P^1 \x \P^1$. In the $\Z$ graded setting the supernatural bundles appear as a subset of natural vector bundles. Our contribution here is theorem \ref{thm:main} which demonstrates the existence of a large number of bi-graded natural vector bundles. In light of theorem \ref{thm:main} the entirety of conjecture \ref{conj} now seems within reach. This would show that all the extremal rays in the bigraded version of $C_{ES}(\P^1 \x \P^1)$ can be represented by a vector bundle. A criterion for when a bi-graded natural vector bundle generates an extemal ray in $C_{ES}(\P^1 \x \P^1)$ is stated in \cite[5.2]{EisenbudMR2810424}. An explicit algorithm to express a general point as a sum of extremal rays would be obvious directions of future work as well as the question of how to `lift' as in \cite{MR3594287}.  

Finally, we note Boij-S\"oderberg theory has now expanded in many different directions. The interested reader can consult \cite{MR3727502} for a more expansive discussion.

\subsection{Strategy}
Here we discuss the perspective we used to prove theorem \ref{thm:main}. A fundamental tool in our approach is the moduli stack $\bun_{G}$ where $G = GL_r$. This stack parameterizes rank $r$ vector bundles on $\P^1$. The stack $\bun_G$ relates to bundles $E$ on $\P^1 \x \P^1$ in three ways:
\begin{itemize}
\item[(1)] A map $\P^1 \to \bun_G$ yields a bundle $E$.
\item[(2)] Given $E$ its restriction $E|_{\P^1 \x pt}$ is a point of $\bun_G$.
\item[(3)] The torsion free part of $Rp_*E$ is point $F$ of $\bun_G$ and the rank of $F$ is determined by $H^*(E|_{\P^1 \x pt})$.
\end{itemize}

Such statements hold on any surface that is the product of smooth projective curves. But the situation is particularly nice for $\P^1$. The splitting theorem \footnote{This result, often attributed to Grothendieck (1957) and Birkhoff (1913), goes back to at least Dedekind and Weber \cite{DedekindMR1579901} (1882). See \cite[Chap.\;1,2.4]{MR2815674} for a short discussion.} for bundles on $\P^1$ assets that any rank $r$ vector bundle $E$ on $\P^1$ is a direct sum of line bundles: $F \cong \bigoplus_{i = 1}^r \cO(n_i)$. In turn the splitting theorem allows for explicit computations to be made over $\bun_G$.

The basic idea is to consider vector bundles $E$ on $\P^1 \x \P^1$ that arise from maps $f \in \hom(\P^1,\bun_G)$ such that $\im f$ contains a single point. We call these bundles {\it constant bundles} but constant bundles are in general {\it not} the pull back of a vector bundle on $\P^1$ because of the stack nature $\bun_G$. We construct a specific class of constant bundles that depend on data $(F,F_1,F_2,\eta)$ where $F,F_1,F_2$ are bundles on $\P^1$ and $\eta$ is an element of $Ext^1(F_2,F_1)\ox H^0(\P^1, \cO(1))$. Then the resulting rank $r$ bundle $E$ satisfies (proposition \ref{p:constantBundlesChi}(d)):
\begin{equation*}\label{eq:EulerCharacteristic}
\frac{\chi(E(x,y))}{r} =  \left(x+\frac{\chi(F)}{r}\right)\left(y+\frac{\chi(F_1\oplus F_2)}{r}\right)-\frac{\deg (F_1^\vee \ox F_2)}{r^2}
\end{equation*}
Given this formula it is possible to show constant bundles exists with prescribed Hilbert polynomial. Then we show (proposition \ref{p:step3}) that by a generic choice of $\eta$ one can ensure that the pushforward of $E$ along $Q \to \P^1$ has natural cohomology.
 
Section \ref{s:prelims} presents basic notation and gives a brief of account of stacks. It also discusses standard results about the stack $\bun_G$ which we utilize in the proof of theorem \ref{thm:main}. Section \ref{s:pushforwardQ} describes the constraints imposed by natural cohomology on the pushforward of a vector bundle along a projection $Q \to \P^1$. These constraints show that in most cases of conjecture \ref{conj} the vector bundle must be of a particular form which we call a constant bundle. In section \ref{s:constant_bundles} various properties of constant bundles are proved leading up to a proof of theorem \ref{thm:main}.

Conjecture \ref{conj} arose from Boij-S\"oderberg through the duality discovered by Eisenbud and Schreyer between the cone of Betti diagrams and the cone of cohomology tables for vector bundles on projective space. This duality is valuable because it opens the door for geometric tools such as moduli stacks to be used. It seems questions in Boij-S\"oderberg theory can suggest, via duality, interesting and perhaps unexpected questions and results in geometry. We hope the geometric techniques here provide a different perspective that can be used on related problems.

\subsection{Notation}
We work over $\ec \C$ and write $\P^1 = \proj \C[z_0,z_1]$ and take $z = \frac{z_1}{z_0}$ as a coordinate on $\P^1$. On $Q = \P^1 \x \P^1$ we think of the second copy as $\proj \C[w_0,w_1]$ and take $w = \frac{w_1}{w_0}$ as a coordiante. If $\mathcal{F} \in \mbox{Coh}(Q)$ we denote $\mathcal{F}(n,m) = \mathcal{F}\ox \cO_Q(n,m)$. We often write $\cO$ instead of $\cO_{\P^1}$ or $\cO_Q$ when no confusion is likely to arise. As usual, $h^i = dim H^i$. Generally $E$ is a vector bundle on $Q$ and $F$ is a vector bundle on $\P^1$.

To a vector bundle $E$ on $Q$ we associate constants $\a_E,\beta_E, \gamma_E \in \Q$ to $E$ via
\begin{equation}\label{eq:cxcyc0}
\frac{\chi(E(x,y))}{\rank E} = (x + \a_E)(y+\beta_E) - \gamma_E.
\end{equation}
The category of schemes over $\C$ is denoted $Sch_\C$ and $Grpo$ denotes the category of groupoids.
 
\section{Stacks and $\bun_G$}\label{s:prelims}
We use the language of stacks in a relatively minimal way. Indeed the reader willing to take for granted some results about moduli stacks can understand the key ideas without much loss. Here we give some remarks on stacks which should be sufficient to follow the paper. An introductory treatment that is more than sufficient for our purposes is \cite{FantechiMR1905329}. We also discuss the stack $\bun_G$ which plays a fundamental role in our proof.

A scheme $X$ can, through its functor of points, be thought of as a functor $X\colon Sch_\C \to Set$ where $Sch_\C$ is the category of $\C$-schemes and for a scheme $T$ we have $X(T) = \hom(T,X)$. There are necessary and sufficient conditions on a functor $X$ that guarantee it is representable by a scheme: $X$ must be a sheaf for the Zariski topology and $X$ must be covered by open sub functors that are themselves representable .

Stacks can be described similarly. Given a functor $\mathcal{X} \colon Sch_\C \to Grpo$ there is a list of additional properties on $\mathcal{X}$ that turn it into a stack. In \cite{FantechiMR1905329} these are called (1) descent data is effective, (2) isomorphisms form a sheaf, and (3) there is a scheme $X$ and an (\'etale or smooth) surjective morphism $X  \to \mathcal{X}$. For this paper it's largely enough to think of stacks as functors into Groupoids. 

A basic example is that of a quotient stack. Let $H$ be an algebraic group acting on a scheme $X$. Define $X/H \colon Sch_\C \to Grpo$ by letting $X/H(T)$ be the groupoid of pairs $\langle (P,\alpha) \rangle$ where $P \to T$ is a principal $H$ bundle and $\alpha \colon P \to X$ is an $H$-equivariant map. One can form the associated fiber bundle $P\x^H X = P \x X/\sim$ where $(p h, x) \sim (p,h x)$. The scheme $P\x^H X$ is a fiber bundle over $T$ with fiber $X$. Then the $H$-equivariant map $\alpha$ is equivalent to a section $\sigma: T \to P\x^H X$. Whenever a smooth group $H$ acts on a smooth scheme $X$ we have $\dim X/H = \dim X - \dim H$. In particular $\dim pt/H = -\dim H$. 

In general we are interested in the moduli stack of $H$-bundles over a fixed scheme $X$. This stack can be conveniently defined as $\bun_H(X):= \hom(X,pt/H)$. More directly $\bun_H(X)(T)$ is the groupoid of principal $H$-bundles on $X \x T$. Over $X \x \bun_H(X)$ there is a universal principal bundle $P^{univ}$. The functorial perspective is then to say that $f \in \hom(T,\bun_H(X)) = \bun_H(X)(T)$ yields a map 
\begin{equation}\label{eq:Puniv}
X\x T \xrightarrow{(id,f)} X \x \bun_H(X)
\end{equation}
and any principal $H$ bundle on $X \x T$ arises as $(id,f)^*P^{univ}$ for some $f$. We are primarily interested in the case $X = \P^1$ and in this case we suppress $\P^1$ and simply write $\bun_H$.

\begin{rmk}\label{rmk:GL_rBundVectorBund}
When $G = GL_r$ the groupoid of principal $G$ bundles is equivalent to groupoid of rank $r$ vector bundles where morphisms are given by vector bundle isomorphisms. The equivalence is given by sending a principal bundle $P$ to the vector bundle $P(\C^r):= P\x_G \C^r$. 

Because of this equivalence we will think of $G$-bundles as vector bundles and sometimes tacitly identify $P$ and $P(\C^r)$.  
\end{rmk}

If we need to emphasize the underlying set of a stack $\mc{X}$ we write $\{\mc{X}\}$. Thus $X/H$ is the quotient stack and $\{X/H\}$ is the set of $H$ orbits in $X$.

\subsection{$\bun_G$, Automorphisms and Extensions}
Let $H$ be an affine algebraic group. Consider the basic cover by affines of $\P^1$: $\P^1 = U_- \cup  U_+$ with $U_- =\ec \C[z^{-1}]$, $U_+ = \ec \C[z]$ and $U_\pm = U_-\cap U_+$. We use this cover to construct three ind-algebraic groups: $H[z],H[z^{-1}],H[z^\pm]$.

Briefly an ind-scheme is an increasing union of schemes $\cup_{i\geq 0} X_i$ with $X_i \to X_{i+1}$ a closed immersion. An ind-algebraic group is a group object in the category of ind-schemes. We define $H[z],H[z^{-1}],H[z^\pm]$ through their functor of points on the subcategory $\Aff_\C$ of affine $\C$-schemes or equivalently the category  of $\C$-algebras.

As functors $H[z],H[z^{-1}],H[z^\pm] \colon \Aff_\C \to Set$ we have
\begin{align*}
H[z](R):= H(R[z]) && H[z^{-1}](R):= H(R[z^{-1}]) &&  H[z^\pm](R):= H(R[z^\pm]) 
\end{align*}
  
The group $H[z^\pm]$ gives transition functions for principal $H$ bundles on $\P^1$.

 Indeed set $P_- =U_- \x H$ and $P_+ = U_+ \x H$. Any $\gamma(z) \in H(\C[z^\pm])$  determines an isomorphism
\begin{equation}\label{eq:gamma}
\gamma \colon P_+|_{U_\pm} \to P_-|_{U_\pm}\ \ \  \gamma(z,h) := (z, \gamma(z) h)
\end{equation}
We denote by $P_\gamma$ the bundle obtained from gluing $P_-,P_+$ along $\gamma$. Different $\gamma$ can lead to an isomorphic bundle. We can alter $P_+,P_-$ by an automorphism in $H(\C[z])$ and $H(\C[z^{-1}])$ respectively leading to a bundle with transition function
\[
P_+|_{U_\pm} \xrightarrow{\gamma_+^{-1}} P_+|_{U_\pm} \xrightarrow{\gamma} P_-|_{U_\pm} \xrightarrow{\gamma_-} P_-|_{U_\pm}
\] 
\begin{rmk}
Recall (remark \ref{rmk:GL_rBundVectorBund}) when $H = GL_r$ then a principal $H$ bundle is equivalent to a rank $r$ vector bundle. In this case we write $F_\gamma = P_\gamma(\C^r)$ to be the associated vector bundle. 
\end{rmk}

\begin{rmk}\label{rmk:O(n)convention}
Under our convention \eqref{eq:gamma} the bundle $\cO(n)$ corresponds to the transition function $z^{-n}$.
\end{rmk}

The following proposition states the precise connection between the groups above and $\bun_G$.
\begin{prop}\label{p:basicsBunG}
Let $G$ be a finite product $\prod_i GL_{r_i}$. Then
\begin{itemize}
\item[(a)] There is a surjective map $G(\C[z^\pm]) \to \{\bun_G\}$ which sends $\gamma \mapsto F_\gamma$.
\item[(b)] The map in (a) descends to a bijection of sets $G(\C[z^{-1}])\backslash G(\C[z^\pm])/G(\C[z]) \leftrightarrow \{\bun_G\}$.
\item[(c)] $G[z^{-1}] \x G[z]$ acts on $G[z^\pm]$ via $(h_-,h_+)g = h_-g h_+^{-1}$ and the stack quotient leads to an equivalence
\[
\bun_G \cong  G[z^{-1}]\backslash G[z^\pm]/G[z]
\]
\item[(d)] The irreducible components of $\bun_G$ are the substacks $\bun_{G,d}$ of $G$-bundles of degree $d$.
\item[(e)] Suppose $H = Aut(F)$ where $F \in \bun_G(\C)$ then  statement (a) holds for $H$ as well. 
\end{itemize}
\end{prop}
\begin{proof}
To begin note $\bun_{G\x H} \cong \bun_G \x \bun_H$ so we can reduce to the case $G = GL_r$.

For (a) we observe that any bundle is trivial on $U,U_-$. In fact if $T\subset G$ is the group of diagonal matrices then the splitting theorem in fact stays the map $T[z^\pm] \to \{\bun_G\}$ given by $\gamma \mapsto F_\gamma$ is surjective. 

The action of $\gamma \mapsto \gamma_- \gamma \gamma_+^{-1}$ of $G(\C[z^{-1}]) \x G(\C[z])$ on $G(\C[z^{\pm}])$ exactly accounts for the ambiguity of going from $\gamma$ to $F_\gamma$ hence (b).

Part (c) is proved in \cite[3.4]{BeauvilleMR1289330} for $G = SL_r$ but \cite[remark 3.6]{BeauvilleMR1289330} addresses the $GL_r$ case. The only discrepency is that in \cite{BeauvilleMR1289330} the affine Grassmannian $Gr_G$ is used instead of $G[z^\pm]/G[z]$. But these are isomorphic as is shown in \cite[7.4]{KumarMR1923198}. 

Part (d) is proved in \cite{DrinfeldMR1362973}; specifically the remark after proposition 5: as $\bun_{G,d}$ is smooth it is enough to show $\bun_{G,d}$ is connected. There is unique point of $\bun_{G,d}$ whose automorphism group has dimension $G$ and all other points of $\bun_{G,d}$ have automorphism groups of larger dimension. But all connected components of $\bun_G$ have the same dimension $-\dim G$ hence $\bun_{G,d}$ is connected.

Finally for (e) note that by \cite[0.5]{Thadds} $H$ is a connected affine algebraic group. According to \ref{c:AutsandExts} the group $H$ has a Levi decomposition $H = R_H \ltimes U_H$ where $U_H$ is unipotent and $R_H$ is a reductive group with $R_H \cong \prod_i GL_{r_i}$.  Every $H$-bundle reduces to an $R_H$ bundle. Also, as $R_H$ is a subgroup of $H$ we have $R_H(\C[z^\pm]) \subset H(\C[z^\pm])$ therefore we deduce (e) from (a).
\end{proof}
\begin{rmk}
Reversing the direction of $\gamma$ in \eqref{eq:gamma} presents $\bun_G$ as  $G[z]\backslash G[z^\pm]/G[z^{-1}]$.
\end{rmk}

Next we discuss automorphisms and extensions.

\begin{cor}\label{c:AutsandExts}
Let $G = GL_r$ and $\gamma \in G[z^\pm]$ and $F_\gamma$ the associated point of $\bun_G(\C)$. Then 
\begin{itemize}
\item[(a)] $Aut(F_\gamma)$ is the stabilizer of $\gamma$ for the action of $G[z^{-1}] \x G[z]$. Explicitly: 
\begin{align*}
Aut(F_\gamma) &= \{(h_-,h_+) \in G[z^{-1}] \x G[z] \ \big|\  \gamma h_+ \gamma^{-1} \in G[z^{-1}]  \}.
\end{align*}
\item[(b)] If $(s_-,s_+)$ with $s_- = \gamma s_+$ is a section of $F_\gamma$ and $(h_-,h_+) \in Aut(F_\gamma)$ then the automorphism acts by $(h_-,h_+)(s_-,s_+) = (h_-s_-,h_+s_+)$.
\item[(c)] $Aut(F(n)) = Aut(F)$ for any $F \in \bun_G(\C)$.
\item[(d)] $Aut(\cO^r) = GL_r$ and 
\begin{align*}
Aut(\cO(1)^{r_1} \oplus \cO^{r_2}) &= \abcd{GL_{r_1}}{\A_1^{r_1 \x r_2}+z^{-1}\A_2^{r_1 \x r_2}}{}{GL_{r_2}} \x  \abcd{GL_{r_1}}{z\A_1^{r_1 \x r_2}+\A_2^{r_1 \x r_2}}{}{GL_{r_2}}  \\
&=\Delta(GL_{r_1} \x GL_{r_2}) \ltimes (Id + \hom(\cO^{r_2},\cO(1)^{r_1}))
\end{align*}
where $\A_i^{r_2 \x r_1}$ are both the affine space $\hom(\A^{r_2},\A^{r_1})$.  
\item[(e)] Write $\P^1 = \proj \C[z_0,z_1]$. If $F = \cO(1)^{r_1} \oplus \cO^{r_2}$ and $g\in Aut(F)$ then the action of $g$ on $H^i(F(n))$, denoted $\Gamma(g)$, is represented by a matrix of the following form
\[
\abcd{GL_{r_1}}{z_0\A_1^{r_1 \x r_2}+z_1\A_1^{r_1 \x r_2}}{}{GL_{r_2}} 
\]
\item[(f)] The generic point of $\bun_{G,d}$ has natural cohomology.
\end{itemize}
\end{cor}

\begin{proof}
Parts (a)-(e) are a computation left to the reader. We proceed to prove (f). The trivial bundle in $\bun_{G,0}$ is a dense open point with automorphism group $G$. Hence $\dim \bun_{G,0} = \dim \bun_{G,d} = -\dim G = -r^2$. Moreover tensoring with $\cO(1)$ gives an equivalence $\bun_{G,d} \cong \bun_{G,d+r}$ so we can assume $0\leq d < r$. Then consider the bundle $F = \cO(1)^d \oplus \cO^{r-d}$. We observe $\dim Aut(F) = \dim G$ hence $F$ must be the generic point of $\bun_{G,d}$ and it has natural cohomology.
\end{proof}

The final property we need to know about $\bun_G$ is the Harder-Narasimhan stratification. For any smooth curve $C$ there is a stratifcation of $\bun_G(C)$ called the Harder-Narasimhan (HN) stratification which is described in \cite{HarderMR0364254}. The HN stratification pulls back to any family of vector bundles over $C$. In the case of $\P^1$ the HN stratication agrees with the stratification by isomorphism type. We are primarily interested in the most generic stratum. If $E \to X \x \P^1$ is a family of vector bundles on $\P^1$ then we denote by $HN^{top}(X)$ the open subset comprising the top stratum in the HN stratification.

\begin{example}\label{ex:Aut(O(1)+O(-1))}
Set $G = GL_2$ and $\gamma = \left(\begin{smallmatrix} z^{-1} & 0\\ 0 & z \end{smallmatrix} \right)$ then $F_\gamma \cong \cO(1) \oplus \cO(-1)$. Then 
\[
Aut(F_\gamma) = \left\{ \abcd{t_1}{\frac{a}{z^2}+\frac{b}{z}+c}{}{t_2} \x  \abcd{t_1}{a+b z+c z^2}{}{t_2}  \right\} \subset G(\C[z^{-1}]) \x G(\C[z])
\]
$Aut(F_\gamma)$ is the five dimensional group $\mathbb{G}_m \x \mathbb{G}_m \x H^0(\cO(2)) \subset H^0(End(F_\gamma))$.
\end{example}

\begin{example}\label{ex:Ext(0,-n)}
If $n>0$ and $F \in Ext^1(\cO^r,\cO(-n)^s)$ then a transition function for $F$ has the form
\[
\abcd{z^nI_s}{Mat_{s \x r}(z\C[z]/z^n)}{0}{I_r}
\]
where $I_j$ is the $j\x j$ identity matrix. A priori the top right $s \x r$ corner has entries in $\C[z^\pm]$ but by column operations with $GL_{r+s}[z]$ and row operations in $GL_{r+s}[z^{-1}]$ we can make the entries lie in $z\C[z]/z^n$. 
\end{example}
In \v{C}ech cohomology $z\C[z]/z^n$ is naturally identified with $H^1(\cO(-n))$ and this identification further identifies $Mat_{s \x r}(z\C[z]/z^n)$ with $Ext^1(\cO^r,\cO(-n)^s)$.  Even in this simple case the HN stratification is interesting, for example $Ext^1(\cO,\cO(-4)) \cong z \C[z]/z^4$. A general element is of the form $a z + b z^2 + c z^3$. The unique point $a = b = c = 0$ is the trivial extension $\cO(-4) \oplus \cO$. All other points of $a c - b^2 =0$ yield the nontrivial extension $\cO(-1)\oplus\cO(-3)$ and all points satisfying $a c - b^2 \neq 0$ yield $\cO(-2)^2$.

This example is typical in the following sense
\begin{prop}\label{p:generalExtConnectingMap}
Set $F_1 = \cO(-a)^{s_1} \oplus \cO(-a+1)^{s_2}$ and $F_2 = \cO(-b)^{r_1} \oplus \cO(-b+1)^{r_2}$ with $a\geq b$.
\begin{itemize}
\item[(a)] $HN^{top}(Ext^1(F_2,F_1))$ consists of bundles with natural cohomology.
\item[(b)] We have $F \in HN^{top}(Ext^1(F_2,F_1))$ if and only if for all twists the connecting map 
\[
H^0(F_2(n)) \to H^1(F_1(n))
\] 
is either injective or surjective for all $n$; equivalently, the map is of maximal rank. 
\end{itemize}
\end{prop}
\begin{proof}
In general $\dim Ext^1(F_2,F_1)$ will be a sum of $h^1(\cO(n))$ and $h^1(\cO(n\pm 1)$ with $n = b-a$. This can only be nonzero if $n-1\leq -2$ which is equivalent to $b-a<0$ or $a > b$. When $Ext^1(F_2,F_1)=0$ the only extension we get is $F_1 \oplus F_2$ and this has natural cohomology when $a=b$. 

The space $Ext^1(F_2,F_1)$ parameterizes a family of bundles hence arises from a morphism 
\[
\phi\colon Ext^1(F_2,F_1) \to \bun_{G,d}, \ \ \ d = \deg(F_1 \oplus F_2) 
\]
by corollary \ref{c:AutsandExts}(f) to prove (a) it is enough to show $\phi$ is dominant. Let $\zeta_1$ be the generic point of $Ext^1(F_2,F_1)$ and recall $0 \in Ext^1(F_2,F_1)$ is the trivial extension $F_1 \oplus F_2$. Set $\eta_1 = \phi(\zeta_1)$ and $\eta_0 = \phi(0)$. 

Then there is finite type locally closed substack $\mathcal{Z} \subset \bun_G$ such that $\phi\colon Ext^1(F_2,F_1) \to \mathcal{Z}$ is surjective; we can take $\mathcal{Z}$ to be the union of the finitely many HN strata that appear in $Ext^1(F_2,F_1)$.

We have $\zeta_0 = \phi^{-1}(\eta_0)$ and so by \cite[6.1.4]{EGAIV},
\begin{equation}\label{eq:EGA}
\codim(\zeta_0) \leq \codim_{\mathcal{Z}}(\eta_0) \leq \codim_{\bun_{G,d}}(\eta_0)
\end{equation}
Moreover $\codim(\zeta_0) = \dim Ext^1(F_2,F_1)$ and $\codim_{\bun_{G,d}}(\eta_0) = -\dim G -(-\dim Aut(\eta_0))$. Setting $r = r_1+r_2$ and $s = s_1 + s_2$ we see that
\begin{align*}
 \dim Ext^1(F_2,F_1) &= (a-b-1)rs +s_1r_2-s_2r_1 =  \codim \zeta_0 \\
 \dim Aut(\eta_0) &= r^s +s^2 + rs(a-b+1) +s_1r_2-s_2r_1\\
 \codim_{\bun_{G,d}}(\eta_0) &= -(r+s)^2 - \dim Aut(\eta_0) = \codim \zeta_0
\end{align*}
It follows that in equation \eqref{eq:EGA} all the inequalities are equalities and $\dim \mathcal{Z} = \dim \bun_{G,d}$, and $\phi$ is dominant.

We give a second proof by showing directly that the generic point $\eta \in \bun_{G,d}$ appears in $Ext^1(F_2,F_1)$. Then $HN^{top}(Ext^1(F_2,F_1))$ must consist of bundles of this isomorphism type. 

We have $\bun_{G,d}$ is connected so there is a vector bundle $E$ over $\P^1 \x \ec \C[[t]]$ such that $E|_{\P^1_{\C((t))}}$ is the generic point of $\bun_{G,d}$ and $E|_{\P^1_{\C}} \cong F_1\oplus F_2$. By \cite{DrinfeldMR1362973} $E$ has a Borel reduction, after an etale base change on $\C[[t]]$. Then $E$ has a transition function of the following form
\[
\gamma(z,t) = \abcd{\gamma_1(z,t)}{\#}{0}{\gamma_2(z,t)}
\] 
Set $d_i = \deg(F_i)$; recall $s = \rank(F_1)$ and $r = \rank(F_2)$. We know $F_{\gamma_i(z,0)} \cong F_i$ which are the generic points of $\bun_{GL_r,d_1}$,$\bun_{GL_s,d_2}$ respectively. It follows that the $F_{\gamma_i(z,t)}|_{\P^1_{\C((t))}}$ cannot be more generic. In particular
\begin{align*}
F_{\gamma_1(z,t)}|_{\P^1_{\C((t))}} &\cong F_1 \in \bun_{GL_s,d_1}\bigg(\C((t))\bigg)\\
F_{\gamma_2(z,t)}|_{\P^1_{\C((t))}} &\cong F_2 \in \bun_{GL_r,d_2}\bigg(\C((t))\bigg).
\end{align*}
Thus $E|_{\P^1_{\C((t))}} =F_{\gamma(z,t)}|_{\P^1_{\C((t))}}  \in Ext^1(F_2,F_1)\ox\C((t))$ and so (a) follows because by corollary \ref{c:AutsandExts}(f) the generic point $E|_{\P^1_{\C((t))}} \in \bun_{G,d}$ has natural cohomology. 

Part (b) follows from (a) via the long exact sequence in cohomology for $0\to F_2 \to F \to F_1 \to 0$.
\end{proof}
\begin{rmk}\label{rmk:degFF}
Note that $\deg(F_1^\vee \ox F_2) = sr(a-b)+s_1r_2-s_2r_1$. As $s\geq s_i$ and $r\geq r_i$ we have $sr \geq s_1r_2$. So if $(b-a)\geq 1$ then $sr(b-a)\geq s_1r_2$ so $0 \geq sr(a-b)+s_1r_2 - s_2r_1 = \deg(F_1^\vee \ox F_2)$. It follows that if $\deg(F_1^\vee \ox F_2)>0$ then $(b-a)<1$ or equivalently $a\geq b$. Thus proposition \ref{p:generalExtConnectingMap} applies when $\deg(F_1^\vee \ox F_2)>0$. We use this in section \ref{s:constant_bundles}.
\end{rmk}

\section{Pushforwards and Natural Cohomology}\label{s:pushforwardQ}
Let $E$ be a vector bundle on $Q$ and let $\P^1 \xleftarrow{p} Q\xrightarrow{q} \P^1$ denote the projection to the first and second factors respectively.

The primary purpose of this section is to prove the following 
\begin{prop}\label{p:precursor_to_constant_bundles}
Let $E$ be a vector bundle with natural cohomology and
\[
\frac{\chi(E(x,y))}{\rank E} = (x + \a_E)(y+\beta_E) - \gamma_E.
\]
If $\a_E \not \in \Z$ and $r = \rank E$ then  
\begin{itemize}
\item[(a)] $Rq_*E(n,m)$ is a vector bundle on $\P^1$ with natural cohomology.
\item[(b)] The fibers $F_z:= E|_{q^{-1}(z)}$ have natural cohomology.
\item[(c)] $F_z \cong F = \cO(b)^{r_1}\oplus \cO(b-1)^{r-r_1}  \forall z \in \P^1$ \\ where $\a_E = \frac{r_1}{r}+b$ with $\frac{r_1}{r} \in (0,1)$ and $b \in \Z$.
\end{itemize}
 The same is true with $q$ replaced by $p$ provided that $\beta_E \not \in \Z$.
\end{prop}

For the proof we need some preliminary results
\begin{lemma}\label{l:spectral_seq}
With $Q,E,p$ as above we have 
\begin{align*}
H^0(Q,E) &= H^0(\P^1,p_*E)\\
 H^1(Q,E) &=  H^1(\P^1,p_*E) \oplus H^0(\P^1,R^1p_*E)\\
 H^2(Q,E) &= H^1(\P^1,R^1p_*E)
\end{align*}
and the same is true for pushforward along $q$.
\end{lemma}
\begin{proof}
Use the Leray spectral sequence $H^i(\P^1,R^jp_*E)\Rightarrow H^{i+j}(Q,E)$. Then $R^2p_*E = 0$ as $\dim p^{-1}(pt) = 1$. The $E_2$ page is
\[
\xymatrix{H^1(\P^1,p_*E) & H^1(\P^1,R^1p_*E) \\ H^0(\P^1,p_*E) & H^0(\P^1,R^1p_*E)}
\]
and all differentials are $0$ so the spectral sequences collapses and the result follows. 
\end{proof}

Next we need a basic lemma about coherent sheaves on $\P^1$.
\begin{lemma}\label{l:torsion_on_P1}
If $F\in \mbox{Coh}(\P^1)$ and $\chi(F(a))$ is constant for $a>>0$ then $F$ is torsion. In particular if $\chi(F(a)) = 0$ for $a>>0$ then $F=0$.
\end{lemma}

\begin{proof}
Any $F\in \mbox{Coh}(\P^1)$ splits as sum of vector bundle and a torsion sheaf: $F = F^{free}\oplus F^{tor}$. Then $H^*(\P^1,F)$ is determined by the splitting type of $F^{free}$ and the length of $F^{tor}$. In particular $\chi(F^{free}(a)) = \chi(F(a)) - \chi(F^{tor})$.

By assumption, for $a>>0$ we have that $\chi(F^{free}(a)) = H^0(F^{free}(a))$ is constant. However if $F^{free}$ is nontrivial then $\chi(F^{free}(a))$ would grow linearly in $a$. Hence $F^{free} = 0$.
\end{proof}

Let $E$ be a vector bundle on $Q$. Consider the vertical twists $E(m,*):= E(m,n)$ for $n \in \Z$ and the horizontal twists $E(*,n):= E(m,n)$ for $m\in \Z$. Vertical and horizontal twists help explain the non integrality condition in proposition \ref{p:precursor_to_constant_bundles}. 

To see this assume $E$ has natural cohomology. Then the parabola $\chi(E(x,y)) = 0$ divides the plane into the three regions where $H^0,H^1,H^2$ are nonzero respectively; the regions are denoted $H^0R$,$H^1R$,$H^2R$. $H^1R$ is uniquely specified by those points where $\chi(E(x,y))<0$. $H^0R$ is the connected region where $\chi(E(x,y))>0$ and which contains points $(x,y)$ with $x>>0$ or $y>>0$ while $H^2R$ is the connected region $\chi(E(x,y))>0$ which contains points $(x,y)$ with $x<<0$ or $y<<0$.

Writing 
\begin{equation}\label{eq:chi_constants}
\frac{\chi(E(x,y))}{\rank E} = (x + \a_E)(y+\beta_E) - \gamma_E \ \ \ \gamma_E >0
\end{equation}
 we see the horizontal and vertical lines $\{y = -\beta_E\}$ and  $\{x = -\a_E\}$ are the only horizontal or vertical lines that are entirely contained in $H^1R$ . A general horizontal or vertical line will meet two regions. We write $E(m,*)\subset H^iR \cup H^jR$ to mean the only nonzero cohomology group is either $H^i(E(m,*))$ or $H^j(E(m,*))$. 

Summarizing we have
\begin{lemma}\label{l:crossing_H*regions}
Let $E$ have natural cohomology and write $\frac{\chi(E(x,y))}{\rank E}$ as 
\[
\frac{\chi(E(x,y))}{\rank E} = (x + \a_E)(y+\beta_E) - \gamma_E \ \ \ \gamma_E >0.
\]
Suppose $\a_E,\beta_E \not \in \Z$ then
\begin{itemize}
\item[(a)] If $(m+\a_E)>0$ then $E(m,*) \subset H^0R \cup H^1R$ 
\item[(b)] If $(m+\a_E)<0$ then $E(m,*)\subset H^2R\cup H^1R$.
\item[(c)] if $(n+\beta_E)>0$ then $E(*,n)\cup H^0R \cup H^1R$.
\item[(d)] if $(n+\beta_E)<0$ then $E(*,n) \subset H^2R\cup H^1R$.
\end{itemize}
\end{lemma}
\begin{proof}
The statement that each line is only contained in two regions follows because any horizontal or vertical line intersects the hyperbola $\chi(E(x,y)) = 0$ at most once; the other intersection point happening at infinity. 

Let us prove (b). The condition $(m+\a_E)<0$ implies that $\chi(E(m,a))>0$ for $a<<0$ thus $E(m,a) \in H^2R$ for $a<<0$. Similarly $\chi(E(m,a))<0$ for $a>>0$ so $E(m,a) \in H^1R$ for $a>>0$. Thus $E(m,*)\subset H^1R\cup H^2R$. The remaining statements are proved in a completely analogous manner.
\end{proof}

The following corollary explains the non integrality condition in proposition \ref{p:precursor_to_constant_bundles}.
\begin{cor}\label{c:non_integrality_condition}
Suppose $E$ is a vector on $Q$ and 
\[
\frac{\chi(E(x,y))}{\rank E} = (x + \a_E)(y+\beta_E) - \gamma_E, \ \ \ \gamma_E>0.
\]
\begin{itemize}
\item[(a)] If $\a_E,\beta_E\in \Z$ then $Rq_*E(-\a_E,*)$ and $Rp_*E(*,-\beta_E)$ are torsion sheaves.
\item[(b)] If $E$ has natural cohomology and $\a_E \not \in \Z$ then $Rq_*E(m,*)$ is a vector bundle for all $m$.
\item[(c)] If $E$ has natural cohomology and $\beta_E \not \in \Z$ then $Rp_*E(*,n)$ is a vector bundle for all $n$.
\end{itemize}
\end{cor}
\begin{proof}
Part (a) follows immediately from lemma \ref{l:torsion_on_P1}. 

For part (b) set $F = Rq_*E(m,0)$. By the projection formula we have
\[
Rq_*E(m,a) = Rq_*\left(E(m,0)\ox q^*\cO(a)\right) = F(a)
\]
As $\a_E \not \in \Z$ we have $(m+\a_E)\neq 0$; assume $(m+\a_E)>0$ then by lemma \ref{l:crossing_H*regions} $H^2(E(m,*))=0$. Moreover $H^1(E(m,a)) = 0$ for $a>>0$ and by lemma \ref{l:spectral_seq} these conditions say  $h^0(R^1q_*E(m,0)\ox \cO(a))=0$ and $h^1(R^1q_*E(m,0)\ox \cO(a))=0$ for $a>>0$ hence $R^1q_*E(m,0)=0$ by lemma \ref{l:torsion_on_P1}. Thus $F = q_*E(m,0)$. 

We need to rule out that $F$ has torsion. If it did then $H^0(F(a))\subset H^0(E(m,a))$ would always be nonzero which contradicts the fact that for $a<<0$ we only have $H^1(E(m,a))$.

The case $(m+\a_E)<0$ is proved similarly as is part (c). 
\end{proof}


Before the proof of proposition \ref{p:precursor_to_constant_bundles} we need one more lemma.
\begin{lemma}\label{l:sameCohomsameIsotype}
Suppose $F_1,F_2$ are vector bundles on $\P^1$ of the same rank. Then $H^*(F_1(n)) = H^*(F_2(n))$ for all $n\in \Z$ if and only if $F_1 \cong F_2$.
\end{lemma}
\begin{proof}
Isomorphic bundles have the same cohomology. For the other direction write $F = \oplus_{i = 1}^r \cO(n_i)^{m_i}$ with $n_i< n_{i+1}$. Then we determine $n_r$  as $n_r = \max\{m | h^0(F(-m))\neq 0\}$ and then $m_r = h^0(F(-n_r))$.  Thus we can pass to $F' = F/\cO(n_r)^{m_r}$ and we are done by induction on the rank of $F$.
\end{proof}

\begin{proof}[proof of proposition \ref{p:precursor_to_constant_bundles}]
By corollary \ref{c:non_integrality_condition} $Rq_*E(n,m)$ is a vector bundle and if it didn't have natural cohomology then lemma \ref{l:spectral_seq} would contradict that $E$ has natural cohomology. This proves (a).

The fiber $F_z:= E|_{q^{-1}(z)}$ determines the rank of $Rq_*E$ at $z \in \P^1$. In particular by cohomology and base change \cite[III.12.11]{HartMR0463157} we have an isomorphism
\begin{equation}\label{eq:cohom_base_change}
Rq_*E(m,n)\ox \cO_{\P^1,z}/m_z \cong  H^*(z \x \P^1,F_z(n)).
\end{equation}
As $Rq_*E(m,n)$ is concentrated in one degree we see that $F_z(n)$ has only $H^0$ or only $H^1$ hence (b).

For part (c) we first prove the fibers $F_z$ have constant isomorphism type. Equation \ref{eq:cohom_base_change} together with the fact that $Rq_*E(m,n)$ is locally free shows
\[
\rank Rq_*E(m,n) = \dim H^*(z \x \P^1,F_z(n)).
\]
Thus by lemma \ref{l:sameCohomsameIsotype} all fibers $F_z$ are isomorphic.

Next we observe that the only vector bundles on $\P^1$ that have natural cohomology are of the form $\cO(b)^{r_1}\oplus \cO(b-1)^{r-r_1}$. In fact given $r>0$ and an integer $d$ there is a unique integer $b$ such that $0\leq d-r b<r$ then setting $r_1 = d-r b$ we get that $F = \cO(b)^{r_1}\oplus \cO(b-1)^{r-r_1}$ is the unique vector bundle on $\P^1$ with natural cohomology such that $\chi(F)=r_1(b+1)+(r-r_1)b = r b+r_1= d$.

Finally, by lemma \ref{l:chiEabformula} we have that  
\[
\chi(F) = \chi(E|_{q^{-1}(z)}) = r \cdot \a_E ,\ \ \ r = \rank E.
\]
Thus by the previous paragraph $F = \cO(b)^{r_1}\oplus \cO(b-1)^{r-r_1}$ where $r \a_E = r_1 + r b$.
\end{proof}

\begin{lemma}\label{l:chiEabformula}
If $E$ is a vector bundle of rank $r$ on $Q$ then for $z,w\in \P^1$
\[
\chi(E(a,b)) = r ab + \chi(E|_{p^{-1}(z)})a + \chi(E|_{q^{-1}(w)})b + \chi(E)
\]
\end{lemma}
\begin{proof}
Consider the short exact sequence
\[
0\to\cO(-a,0) \to \cO \to \cO_{a D}\to 0
\]
where $D = p^{-1}(z)$. Then applying $\ox E$ and taking $\chi$ gives $\chi(E(a,0) )= a\cdot \chi(E|_{p^{-1}(pt)})+\chi(E)$. We get a similar expression for $\chi(E(0,b))$. Doing the same with
\[
0\to \cO(-a,-b) \to \cO(-a,0) \oplus \cO(0,-b) \to \cO \to \cO_{a b \cdot pt} \to 0
\]
yields $\chi(E(a,b)) = r a b+\chi(E(a,0))+\chi(E(0,b))-\chi(E)$ which with the previous expressions give the result.
\end{proof}

\section{Constant Bundles}\label{s:constant_bundles}
Let $E$ be a vector bundle on $Q$. Recall $\P^1 \xleftarrow{p} Q \xrightarrow{q} \P^1$ are the two projections. 

We say $E$ is {\it  constant} if the isomorphism type of $E|_{p^{-1}(z)}$ is constant independent of $z$ or if the same condition holds for $E|_{q^{-1}(z)}$. In the first case we say $E$ is constant with respect to $p$ and in the latter case $E$ is constant with respect to $q$. To unify these cases we write $Q = \P^1_b \x \P^1_f$ (base and fiber) and a constant bundle $E$ on $Q$ has the property that $F = E|_{z\x \P^1_f}$ is constant as $z \in \P^1_b$ varies.

Loosely speaking bundles over $Q$ are more complicated than bundles over $\P^1$. Constant bundles, however, are equivalent to principal bundles over $\P^1$ as we shall see shortly.

\begin{lemma}\label{l:constantBundles}
A rank $r$ constant bundle $E$ on $Q$ is determined by a pair $(F,P)$ where $F$ is a rank $r$ vector bundle on $\P^1$ and $P$ is a principal $Aut(F)$-bundle. In particular a morphism $\lambda\colon \Gm \to Aut(F)$ specifies such a bundle $P$.
\end{lemma}
\begin{rmk}
In general $\lambda$ is not a group homomorphism. 
\end{rmk}

\begin{proof}
We can view $E$ as a family of vector bundles on $\P^1_f$ parametrized by $\P^1_b$. That is, we can say $E$ arises from a map $\phi :\P^1_b \to \bun_G(\C)$ with $G = GL_r$. 

The image of $\phi$ necessarily consists of a single point $F \in \bun_G(C)$; that is $\phi$ must factor through the locally closed substack $pt/Aut(F) \subset \bun_G$. Then a morphism $\P^1_b \to pt/Aut(F)$ is by definition a principal $Aut(F)$ bundle on $\P^1_b$. The last statement follows from proposition \ref{p:basicsBunG}(e).
\end{proof}  

Recall the universal bundle $P^{univ}$  (see \eqref{eq:Puniv}) of $\bun_G \x \P^1_f$. From the proof of the lemma we see that any constant bundle $E$ is of the form $(\phi,id)^*P^{univ}$ for a map $Q \xrightarrow{(\phi,id)} \bun_G \x \P^1_f$.

\begin{lemma}\label{l:EequalsP(F)}
Let $F \in \bun_G(\C)$ and set $H = Aut(F)$. The restriction of $P^{univ}$ to  $pt/H \x \P^1_f$ is $F/H$. Let $E$ be a constant bundle arising from a pair $(F,P_H)$ where $P_H$ is a principal $H$-bundle. Then $E = P_H\x^H F$. 
\end{lemma}
\begin{proof}
A vector bundle on $pt/H\x \P^1_f$ is equivalent to an $H$-equivariant vector bundle thus $P^{univ}|_{ pt/H \x \P^1_f}$ is  $F$ endowed with an $H$-equivariant structure given by the natural action of $H$ on $F$. Finally from the diagram 
\[
\xymatrix{ P_H \x \P^1_f  \ar[r]^{(\tilde \phi,id)}\ar[d] & pt \x \P^1_f \ar[d] & F \ar[l] \\
\P^1_b \x \P^1_f \ar[r]^{(\phi,id)} & pt/H \x \P^1_f  & \ar[l] F/H
}
\]
it follows that $(id,\phi)^*F/H = \frac{(id,\tilde \phi)^*F}{H} =( P_H \x F)/H=  P_H \x^H F$.
\end{proof}

Let us use coordinates $(z,w)$ on $Q$ to make this construction more explicit. Assume $E = (\phi, id)^*P^{univ} = P \x^H F$. Let $\gamma \in G[z^\pm]$ be a transition function for $F$. Recall $H = Aut(F) \subset G[z^{-1}] \x G[z]$. The $H$-bundle $P$ is described by a transition function $(\a_-,\a_+)\in H[w^\pm]$. Finally, if $G$ has rank $r$ then $E$ is glued together as follows:
\begin{equation}\label{eq:E transition functions}
\xymatrix{
\C[z^{-1},w]^{\oplus r} \ar[r]^{\a_-} & \C[z^{-1},w^{-1}]^{\oplus r}\\
\C[z,w]^{\oplus r} \ar[u]^\gamma\ar[r]^{\a_+} & \C[z,w^{-1}]^{\oplus r} \ar[u]_\gamma
}
\end{equation} 
Note that $\a_-,\a_+ \in G[z^\pm,w^\pm]$ while $\gamma$ only depends on $z$.

\subsection{Pushforwards}
Here we construct constant bundles with respect to $q \colon Q \to \P^1$ and compute their Euler characteristic. In this convention $\P^1_b = q(Q)$ and $\P^1_f = q^{-1}(z) = \P^1 \x pt$.

We begin with this lemma
\begin{lemma}\label{l:basic q_*E}
Let $F = \cO_{\P^1_f}^{r_1} \oplus \cO_{\P^1_f}^{r_2}(-1)$ and set $s_1 = h^0(F(n))$ and $s_2 = h^1(F(n))$. Let $P$ be an $Aut(F)$-bundle and set $E = P \x^{Aut(F)} F$ and $V_\pm = \ec \C[w^\pm]$. Then
\begin{itemize}
\item[(a)] $H^0(E(n,0)|_{\P^1 \x V_\pm}) \cong \C[w^\pm]^{\oplus s_1}$ and $H^1(E(n,0)|_{\P^1 \x V_\pm}) \cong \C[w^\pm]^{\oplus s_2}$
\item[(b)] The automorphism $(\a_-,\a_+)$ from equation \ref{eq:E transition functions} acts on $H^i(E(n,0)|_{\P^1 \x V_\pm})$ via corollary \ref{c:AutsandExts} giving elements $\lambda(n) \in GL_{s_1}[w^\pm]$ and $\lambda^1(n) \in GL_{s_2}[w^\pm]$.
\item[(c)] $q_*E(n,0) = F_{\lambda(n)}$ for $n\geq 0$ and $R^1q_*E(n,0) = F_{\lambda^1(n)}$ for $n\leq -1$ where $F$ is the map in proposition \ref{p:basicsBunG}(a).  
\end{itemize}
\end{lemma}
\begin{proof}
The restriction $E|_{\P^1 \x V_\pm}$ corresponds to a map $V_\pm \to pt/Aut(F)$ which must be trivial since every $Aut(F)$ bundle on $V_\pm$ is trivial. Thus $E(n,0)|_{\P^1 \x V_\pm} \cong F(n)\ox \cO_{V_\pm}$. Hence $H^i(E(n,0)|_{V_\pm}) \cong H^i(F(n)) \ox \C[w^\pm]$ hence (a).

Part (b) follows readily from an examination of equation \ref{eq:E transition functions}. 

Finally (c) follows from the fact that $H^1(F(n)) = 0$ when $n\geq 0$ and $H^0(F(n)) = 0$ when $n\leq -1$. 
\end{proof}
Examples of the matrices $\lambda(n),\lambda^1(n)$ are computed in equation \eqref{eq:pushforwardTransitionFuncs}.

Recall the group $Aut(F)$ (corollary \ref{c:AutsandExts}(d)) factorizes as $Aut(F) = L_F \ltimes U_F$ with $L_F = GL_{r_1} \x GL_{r_2}$ and $U_F \cong Id +\hom(\cO(-1)^{r_2},\cO^{r_1})$ is the unipotent radical of $Aut(F)$.

In particular an arbitrary element of $Aut(F)[w^\pm]$ that has the following form:
\begin{equation}\label{eq:lambda}
\lambda(w) = \abcd{\lambda_{r_1}(w)}{\eta(z,w)}{0}{\lambda_{r_2}(w)}, \ \ \ \lambda_{r_i}(w) \in GL_{r_i}(\C[w^\pm]),\ \ \ \ \eta(z,w) \in U_F(\C[w^\pm]).
\end{equation}

Similar to example \ref{ex:Ext(0,-n)} we can use appropriate row and column operations so that 
\begin{equation}\label{eq:finitDimExtSpace}
\eta(z,w) \in Ext^1_b(F_2,F_1) \ox H^0(\cO_{\P^1_f}(1))
\end{equation}
where the subscript $b$ indicates this is $Ext^1$ of two vector bundles $F_1 = F_{\lambda_{r_1}(w)}$ and $F_2 = F_{\lambda_{r_2}(w)}$ on $\P^1_b$; recall (proposition \ref{p:basicsBunG}) that $F_\gamma$ is the vector bundle with transition function $\lambda$. 

The following proposition describes the pushforward and Hilbert polynomial of constant bundles $E$ which are determined by the pair $(F,\lambda(w))$ with $\lambda$ as in \eqref{eq:lambda}. For the statement recall the top Harder-Narasimhan stratum $HN^{top}$ discussed before example \ref{ex:Aut(O(1)+O(-1))}. We also use the following notation: if $g\in Aut(F)$ then let $\Gamma(g)$ denote the action of $g$ on $H^*(F(n))$.

\begin{prop}\label{p:constantBundlesChi}
Set $F=\cO^{r_1} \oplus \cO(-1)^{r_2}$ and $\lambda(w)$ be as in \eqref{eq:lambda}. Let $E = E(F,\lambda(w))$ be the constant vector bundle on $Q$ determined by $F$  and $\lambda(w)$ via lemma \ref{l:constantBundles}.
\begin{itemize}
\item[(a)] $R^1q_*E(n,0) = 0$ if $n\geq 0$ and $q_*E(n,0)=0$ if $n<0$.
\item[(b)] $F_1=q_*E$ and $F_2 = R^1q_*E(-1,0)$ are vector bundles given by $F_1 = F_{\lambda_{r_1}(w)}$ and $F_2 =F_{ \lambda_{r_2}(w)}$.
\item[(c)] For $n\in \Z$ we have $Rq_*E(n,0) \in Ext^1(F_2^{| n|},F_1^{|n+1|})$.
\item[(d)] $\chi(E(x,y))  = r x y + \chi(F_1\oplus F_2)x + \chi(F) y +\chi(F_1)$ with $r = r_1 +r_2$. Equivalently,
\[
\frac{\chi(E(x,y))}{r} =  \left(x+\frac{\chi(F)}{r}\right)\left(y+\frac{\chi(F_1\oplus F_2)}{r}\right)-\frac{\deg (F_1^\vee \ox F_2)}{r^2}.
\]
Alternatively, if $d_1 = \deg(F_1)$ and $d_2 = \deg(F_2)$ then 
\[
\frac{\chi(E(x,y))}{r} =  \left(x+\frac{r_1}{r}\right)\left(y+1+\frac{d_1+d_2}{r}\right)-\frac{d_2 r_1 - d_1 r_2}{r^2}.
\]
\item[(e)] If $F_i$ have natural cohomology and $\deg (F_1^\vee \ox F_2)>0$ then $HN^{top}(Ext^1(F_2^{| n|},F_1^{|n+1|}))$ consists of bundles with natural cohomology.
\item[(f)] The bundle $E$ can be recovered from the splitting type of $F,F_1,F_2$ and the extension data $\eta(z,w) \in Ext^1_b(F_2,F_1) \ox H^0(\cO_{\P^1_f}(1))$.
\end{itemize}
\end{prop}
\begin{proof}
Part (a) follows from the fact that $H^1(F(n)) = 0$ for $n\geq 0$ and $H^0(F(n)) = 0$ for $n<0$.

For part (b) observe that the isomorphism type of $q_*E$ is given by the action of $\Gamma(\lambda(w))$ on $H^0(\cO^{r_1}) = H^0(F)$. Similarly the isomorphism type of $R^1q_*E$ is given by the action of $\Gamma(\lambda(w))$ on $H^1(\cO^{r_2}(-2)) = H^1(F(-1))$.

If $n \neq 0,-1$ then both summands of $F$ contribute to $H^*(F(n))$. In particular the summand $\cO^{r_1}(n) \subset F(n)$ pushes forward to give the subbundle $Rq_*\cO^{r_1}(n) = F_1^{|n+1|} \subset Rq_*E(n,0)$. Then the quotient $Rq_*E(n,0)/F_1^{|n+1|}$ kills the action of $\eta$ so $Rq_*E(n,0)/F_1^{|n+1|} = Rq_*\cO^{r_2}(n-1) = F_2^{|n|}$. This proves (c).

By (b), $q_*E = F_1$ hence $\chi(E) = \chi(F_1)$. As $E$ is constant along $q$ we have $\chi(E|_{q^{-1}(z)}) = \chi(F)$.  It remains to determine $\chi(E|_{p^{-1}(z)})$. But we have 
\[
\chi(E|_{p^{-1}(z)}) + \chi(F_1)=\chi(E(1,0))=2\chi(F_1) + \chi(F_2)
\]
where the first equality is from lemma \ref{l:chiEabformula} and the second from part (c). Therefore $\chi(E|_{p^{-1}(z)}) = \chi(F_1) + \chi(F_2)$ which proves the first part of (d). The second part follows from a direct computation. 

Writing $F_1 = \cO(-a)^{s_1} \oplus \cO(-a+1)^{s_2}$ and $F_2 = \cO(-b)^{r_1} \oplus \cO(-b+1)^{r_2}$ then the condition $\deg(F_1^\vee \ox F_2)>0$ translate to $b-a\geq 0$ by remark \ref{rmk:degFF} so (e) follows from proposition \ref{p:generalExtConnectingMap}(a).
 
Finally, (f) follows from corollary \ref{l:constantBundles} because $F_1,F_2,\eta$ determine the blocks of $\lambda(w)$ in equation \ref{eq:lambda}.
\end{proof}

In light of proposition \ref{p:constantBundlesChi}(f) we denote constant bundles of this form as $E = E(F,F_1,F_2,\eta)$. We call $Ext^1_b(F_2,F_1) \ox H^0(\cO_{\P^1_f}(1))$ the {\it space of extension data} for $E$.

\begin{rmk}\label{rmk:extDataModulispace}
The space $Ext^1_b(F_2,F_1) \ox H^0(\cO_{\P^1_f}(1))$ can be viewed as a moduli space of constant bundles by fixing $F,F_1,F_2$ and associating to $\eta \in Ext^1_b(F_2,F_1) \ox H^0(\cO_{\P^1_f}(1))$ the constant bundle $E = E(F,F_1,F_2,\eta)$. The reason to isolate $\eta$ is that the cohomology of $E$ depends continuously on the class $\eta$ determines in $Ext^1(F_2^{|n|},F_1^{|n+1|})$ whereas the choice $F_1,F_2$ only contribute the discrete choices of $r_i = \rank(F_i)$ and $d_i = \deg(F_i)$. 
\end{rmk}

\begin{cor}\label{c:prescribedChi}
Suppose $\a,\beta,\gamma \in \Q$ and $\a \in (0,1)$ and $p(x,y) = r (x+\a)(y+\beta)-r \gamma \in \Z[x,y]$. Then there is a constant bundle $E = E(F,F_1,F_2,\eta)$ such that $\chi(E(x,y)) = p(x,y)$. Moreover $F_1,F_2$ can be taken to have natural cohomology.
\end{cor}
\begin{proof}
Write $\a = \frac{r_1}{r}$ and set $F = \cO^{r_1}\oplus \cO(-1)^{r_2}$. We wish to find vector bundles $F_1,F_2$ such that for $E = E(F,F_1,F_2,\eta)$ we have $\chi(E(x,y)) = p(x,y)$. But by proposition \ref{p:constantBundlesChi}(d), $\chi(E(x,y))$ depends only $\rank F_i$ and $\deg F_i$. We show there exists $d_1,d_2$ such that for any vector bundles $F_i$ of rank $r_i$ and degree $d_i$ we have $E = E(F,F_1,F_2,\eta)$ has the desired Hilbert polynomial.

By proposition \ref{p:constantBundlesChi}(d) we need to find $r_1,r_2,d_1,d_2$ such that 
\[
(x+\a)(y+\beta)- \gamma = \left(x+\frac{r_1}{r}\right)\left(y+1+\frac{d_1+d_2}{r}\right)-\frac{d_2 r_1 - d_1 r_2}{r^2}.
\]
We must take $r_1 = r \a$ and $r_2 = r - r_1$ and it remains to find $d_1,d_2$ such that
\[
\frac{1}{r^2}\abcd{r}{r}{-r_2}{r_1}\ab{d_1}{d_2}+\ab{1}{0} = \ab{\beta}{\gamma}
\]
which has the unique solution
\[
\ab{d_1}{d_2}  =\ab{ r_1\beta -r\gamma}{r_2\beta+r \gamma}- \ab{r_1}{r_2} = \ab{ r(\a \beta -\gamma)}{r\beta -r (\a\beta - \gamma)}- \ab{r \a}{r - r \a}
\]
As $p(x,y) \in \Z[x,y]$ we have $r \a, r \beta, r(\a \beta -\gamma) \in \Z$ thus $d_1,d_2 \in \Z$. The final statement follows from lemma \ref{l:natcohomP1}
\end{proof}
\begin{rmk}
The hypotheses $\a \in (0,1)$ and $r \a \in \Z$ together require $r>1$. In particular corollary \ref{c:prescribedChi} can fail if $r = 1$: there is no line bundle with $L$ with $\chi(L(x,y)) = x y - 1$.
\end{rmk}

\begin{lemma}\label{l:natcohomP1}
If $r>0$ and $d\in \Z$ then there is a rank $r$ bundle $F$ on $\P^1$ such that $\chi(F) = d$ and $E$ has natural cohomology.
\end{lemma}
\begin{proof}
As $\chi(F(1)) = \chi(F)+r$, we can assume $0\leq d < r$. Then take $F = \cO^d \oplus \cO(-1)^{r-d}$.
\end{proof}

\begin{example}\label{ex:bundleFromPoly}
Let $p = (x+\frac{1}{3})y - 2$. We construct a rank $3$ constant bundle $E$ with $\chi(E(x,y)) = 3 p(x,y)$. Set $F = \cO\oplus \cO(-1)^2$. Let $F_1 = \cO(-7)$ and $F_2 = \cO(2)^2$; in particular $\chi(F_1\oplus F_2) = 0$. Then for any choice of extension data $\eta$ there is, by proposition \ref{p:constantBundlesChi}, a constant bundle with $E = E(F,F_1,F_2,\eta)$ with the desired Hilbert polynomial. 

From $F_1,F_2,\eta$ we recover an element $\lambda(w) \in Aut(F)[w^\pm]$. We compute the automorphism $\Gamma(\lambda(w))$ of $H^*(E|_{\P^1 \x \ec \C[w^\pm]})$ induced by $\lambda(w)$. This allows us to compute $Rq_*E(n,0)$. Let $V(w,-1,6) = \{\sum_{i = -1}^6 c_i w^i\}$ be the vector space of Laurent polynomials supported between $w^{-1}$ and $w^6$.  Then by corollary \ref{c:AutsandExts}(e)
\begin{equation}\label{eq:GammaLambda}
\Gamma(\lambda(w)) = 
\left(\begin{array}{ccc}
w^7& a_0z_0 + a_1z_1 & b_0 z_0 + b_1 z_1\\
& w^{-2} &\\
& & w^{-2}
\end{array}
\right),\ \ a_i,b_i \in V(w,-1,6)
\end{equation}
In fact, there is an isomorphism $Ext^1(\cO(2),\cO(-7)) \cong V(w,-1,6)$ so $V(w,-1,6)^2 \cong Ext^1_b(F_2,F_1)$ and the top right block of equation \ref{eq:GammaLambda} is identified with $Ext_b(F^2,F_1)\ox H^0(\P^1_f,\cO(1))$.

Let $U = \ec \C[w^\pm]$ so that $\P^1 \x \ec \C[w^\pm] = q^{-1}(U)$.  We have $E(n,0)|_{q^{-1}(U)} \cong p^*F(n)|_{q^{-1}(U)} = \cO(n,0) \oplus \cO(n-1,0)|_{q^{-1}(U)}$. The transtion function for $Rq_*E(n,0)$ is determined by the action of $\Gamma(\lambda(w))$ on $H^*(E(n,0)|_{q^{-1}(U)})$. We have
\[
H^0(E(-1,0)|_{q^{-1}(U)}) \cong \left(\begin{array}{c}
0 \\
\frac{\C[w^\pm]}{z_0z_1}\\
\frac{\C[w^\pm]}{z_0z_1}
\end{array}\right),\ \ 
H^0(E(0,0)|_{q^{-1}(U)}) \cong \left(\begin{array}{c}
\C[w^\pm]\\
0\\
0
\end{array}\right)
\]
where $\frac{1}{z_0z_1} \in H^1(\cO(-2))$ is the standard generator. For a third case we have
\[
H^0(E(1,0)|_{q^{-1}(U)}) \cong \left(\begin{array}{c}
z_0\C[w^\pm]\oplus z_1\C[w^\pm]\\
\C[w^\pm]\\
\C[w^\pm]
\end{array}\right)
\]
From equation \ref{eq:GammaLambda} we see that the following matrices
\begin{equation}\label{eq:pushforwardTransitionFuncs}
\left(\begin{array}{cc}
w^{-2} & \\
 & w^{-2}
\end{array}\right), \ \ \left(\begin{array}{c}
w^{7} 
\end{array}\right), \ \ \left(\begin{array}{cccc}
w^7 &  & a_0 & b_0\\
 & w^7 & a_1 & b_1\\
 &  & w^{-2} & \\
 &  &  & w^{-2}
\end{array}\right), \ a_i,b_i \in V(w,-1,6)
\end{equation}
give transition functions for the $R^1q_*E(-1,0), q_*E$ and $q_*E(1,0)$ respectively. So we verify proposition \ref{p:constantBundlesChi}(b),(c): $R^1q_*E(-1,0) = F_2$, $q_*E = F_1$ and $q_*E(1,0) \in Ext^1(F_2,F_1^2)$. 
\end{example}

\subsection{Extension Data}\label{ss:extension data}
At this point we have established that there exist constant bundles $E = E(F,F_1,F_2,\eta)$ with prescribed Hilbert polynomial and if moreover $F_i$ have natural cohomology then $Rq_*E(n,0)$ is always an extension of bundles with natural cohomology: $Rq_*E(n,0) \in Ext^1(F_2^{|n|},F_1^{|n+1|})$.

Recall from remark \ref{rmk:extDataModulispace} that $\eta$ varies in the space $Ext^1_b(F_2,F_1) \ox H^0(\cO_{\P^1_f}(1))$. The final is step is to show there is a choice of extension data $\eta$ such that $Rq_*E(n,0)$ has natural cohomology. We will show a generic choice of $\eta$ leads to a class in $Ext^1(F_2^{|n|},F_1^{|n+1|})$ which lies in the top stratum of the Harder-Narasimhan stratification:
\begin{equation}\label{eq:lastStep}
\mbox{generic }\eta \ \Rightarrow   Rq_*E(n,0)\in HN^{top}(Ext^1(F_2^{|n|},F_1^{|n+1|}))
\end{equation}

Then $Rq_*E(n,0)$ has natural cohomology by proposition \ref{p:generalExtConnectingMap}. We establish \eqref{eq:lastStep} using proposition \ref{p:generalExtConnectingMap}(b). That is we have for every twist $(n,m)$ a short exact sequence 
\begin{equation}\label{eq:sesPushforward}
0\to F_1^{|n+1|}(m) \to Rq_*E(n,m)\to F_2^{|n|}(m)\to 0
\end{equation}
which yields a connecting map $c_\eta(n,m)$ in cohomology:
\[
c_\eta(n,m) \in \hom( H^0(F_2^{|n|}(m)) , H^1(F_1^{|n+1|}(m)))
\]
Then \eqref{eq:lastStep} can be refined to 
\begin{equation}\label{eq:refinedLastStep}
\mbox{generic }\eta \Rightarrow c_\eta(n,m) \mbox{ has maximal rank } \Rightarrow  Rq_*E(n,0)\in HN^{top}(Ext^1(F_2^{|n|},F_1^{|n+1|}))
\end{equation}

The steps are 
\begin{itemize}
\item[(1)] There is an isomorphism $Ext^1_b(F_2,F_1) \ox H^0(\cO_{\P^1_f}(1)) \to \hom( H^0(F_2) , H^1(F_1^{2}))$ which sends $\eta \mapsto c_\eta(1,0)$.
\item[(2)] If $\eta$ is generic then $c_\eta(1,m)$ has maximal rank.
\item[(3)] The map $c_\eta(n,m)$ is built iteratively from $c_\eta(1,m)$ and when $c_\eta(1,m)$ has maximal rank then $c_\eta(n,m)$ has maximal rank.
\end{itemize}
In fact these steps handle the case $n\geq 0$. The remaining cases are proved using the same steps except throughout one replaces $c_\eta(1,0)$ with $c_\eta(-2,0)$.

The standing assumptions for this section are:
\bigskip 
 
\paragraph{We consider a vector bundle $E = E(F,F_1,F_2,\eta)$ where $F = \cO^{r_1}\oplus \cO(-1)^{r_2}$, $F_i$ are of rank $r_i$ and have natural cohomology; $r = r_1 + r_2$. If $\lambda_{r_i}$ are transition functions for $F_i$ then $\lambda_{r_i}$ and $\eta \in Ext^1_b(F_2,F_1) \ox H^0(\cO_{\P^1_f}(1))$ fit together to give an element $\lambda \in Aut(F)[w^\pm]$ as in equation \ref{eq:lambda}. Moreover we assume $\deg F^\vee_1 \ox F_2>0$.
}\label{para:standingAssumptions}
\[
\empty
\]
The following lemma constitutes steps (1),(2):
\begin{lemma}\label{l:steps1,2}
Let $F,F_1,F_2$ be fixed as in \ref{para:standingAssumptions} and for every $\eta \in Ext^1_b(F_2,F_1) \ox H^0(\cO_{\P^1_f}(1))$ let $E_\eta:= E(F,F_1,F_2,\eta)$ and let $c_\eta(n,m)$ be the connecting map in cohomology associated to 
\[
Rq_*E_\eta(n,m) \in Ext^1(F_2^{|n|}(m), F_1^{|n+1|}(m)).
\]
\begin{itemize}
\item[(a)] The linear map $Ext^1_b(F_2,F_1) \ox H^0(\cO_{\P^1_f}(1)) \to \hom( H^0(F_2) , H^1(F_1^{2}))$ which sends $\eta$ to $c_\eta(1,0)$ is an isomorphism. 
\item[(b)] For a generic choice of $\eta$ the connecting map $c_\eta(1,m)$ will have maximal rank for every $m$.
\item[(c)] The linear map $Ext^1_b(F_2,F_1) \ox H^0(\cO_{\P^1_f}(1)) \to \hom( H^0(F_2^2) , H^1(F_1))$ which sends $\eta$ to $c_\eta(-2,0)$ is an isomorphism.
\item[(d)] For a generic choice of $\eta$ the connecting map $c_\eta(-2,m)$ will have maximal rank for every $m$.
\end{itemize}
\end{lemma}
\begin{proof}
The linear map in (a) is injective because if $\eta \neq 0$ then $q_*E(1,0)$ is a nontrivial extension in $Ext^1(F_2,F_1^2)$ and hence $c_\eta(1,0) \neq 0$. So (a) follows because 
\[
\dim Ext^1_b(F_2,F_1) \ox H^0(\cO_{\P^1_f}(1)) = 2 \dim Ext^1(F_2,F_1) = \dim Ext^1(F_2,F_1^2).
\] 
Set $W = HN^{top}(Ext^1_b(F_2,F_1) \ox H^0(\cO_{\P^1_f}(1)))$ to be the preimage of $HN^{top}(Ext^1(F_2,F_1^2))$ under the isomorphism in (a); this is an open set. Then by proposition \ref{p:generalExtConnectingMap}(b) it follows that (b) is satisfied for any $\eta \in W$.

Part (c) is proved similarly to (a), replacing $q_*E(1,0)$ with $R^1q_*(-2,0)$ and then (d) follows from (c) again using proposition \ref{p:generalExtConnectingMap}(b).
\end{proof}

The following is step (3):
\begin{prop}\label{p:step3}
Consider the same hypotheses as in lemma \ref{l:steps1,2} and suppose $\eta\in Ext^1_b(F_2,F_1) \ox H^0(\cO_{\P^1_f}(1))$ is generic in the sense that it satisfies lemma \ref{l:steps1,2}(b). For such an $\eta$ the connecting map
\[
c_\eta(n,m) \colon H^0(F_2^{|n|}(m)) \to H^1(F_1^{|n+1|}(m))
\]
has maximal rank for all $n,m$. In particular, for such a choice of $\eta$ the vector bundle $Rq_*E_\eta(n,0)$ has natural cohomology for all $n$. 
\end{prop}
\begin{proof}
We will prove the case when $n\geq 0$ as the case $n\leq -1$ is completely analogous. For $n=0$ there is nothing to prove and by assumption the result holds for $n=1$.

Consider the case $n=2$. We have $c_\eta(2,m) \in hom(H^0(F_2)^2, H^1(F_1)^3)$ and $c_\eta(2,m)$ restricts to a copy of $c_\eta(1,m)$ on each copy of $H^0(F_2)$. Specifically, $c_\eta(2,m) = c_\eta(1,m)_{1,2} \oplus c_\eta(1,m)_{2,3}$. Where $c_\eta(1,m)_{i,i+1}$ has image in the $i$,$i+1$st copies of $H^1(F_1)$.

For every $m$ the map $c_\eta(1,m)$ is either injective or surjective. Assume the latter. Then together $c_\eta(1,m)_{1,2} \oplus c_\eta(1,m)_{2,3}$ surjects onto $H^1(F_1)^3$. If on the other hand we have that $c_\eta(1,m)$ is injective then the dual map $c_\eta(1,m)^* \colon (H^1(F_1)^2)^* \to H^0(F_2)^*$ is surjective and so in a similar fashion $c_\eta(2,m)^*$ is surjective and thus $c_\eta(2,m)$ is injective. It follows that $c_\eta(2,m)$ is of maximal rank for every $m$.

In general, $c_\eta(n,m) = \bigoplus_{i=1}^{n}c_\eta(1,m)_{i,i+1} = c_\eta(n-1,m) \oplus c_\eta(1,m)_{n,n+1}$ and the same argument applies thus the first part of the proposition is proved.

The final statement follows because by proposition \ref{p:generalExtConnectingMap}(b) the bundle $Rq_*E_\eta(n,0)$ represents a generic point of $Ext^1(F_2^{|n|},F_1^{|n+1|})$ and thus by \ref{p:generalExtConnectingMap}(a) it also has natural cohomology.
\end{proof}

\begin{thm}
\label{thm:main}
Let $p(x,y) = (x+\alpha)(y+\beta)-\gamma \in \Q[x,y]$ with $\gamma>0$ and $\alpha, \beta$ not both integral. If $r>1$ is an integer such that $r p(x,y) \in \Z[x,y]$ then there is a vector bundle $E$ of rank $r$ with natural cohomology such that $\chi(E(a,b)) = r p(a,b)$.
\end{thm}
\begin{proof}
We exhibit a constant bundle $E = E(F,F_1,F_2,\eta)$ with the desired properties. The involution of $Q$ that switches the copies of $\P^1$ reverses the roles of $\a,\beta$ so without loss of generality we can assume $\a \not \in \Z$. 


By corollary \ref{c:prescribedChi} there is constant bundle $E = E(F,F_1,F_2,\eta)$ with $\frac{\chi(E(x,y))}{r} = p(x,y)$. Further, the isomophism type of $F$ is determined by $\a$. By lemma \ref{l:natcohomP1} there is a unique choice for $F_1,F_2$ if we require them to have natural cohomology. It remains to show we can choose $\eta$ such that $E$ has natural cohomology. 

By the projection formula it is enough to show $Rq_*E(n,0)$ has natural cohomology for $n\in \Z$. By proposition \ref{p:constantBundlesChi}(c) we always have that $Rq_*E \in Ext^1(F_2^{|n|},F_1^{|n+1|})$. The condition $\gamma>0$ together with proposition \ref{p:constantBundlesChi}(d) imply $\deg F_1^\vee \ox F_2 >0$. Proposition \ref{p:step3} ensures we can choose $\eta$ such that $Rq_*E(n,0)$ has natural cohomology. Hence the result is proved.
\end{proof}

We end with some remarks about the remaining cases of the conjecture. One must construct vector bundles with natural cohomology such that $\chi(E(x,y)) = r (x y - \gamma)$. On the face of it is not clear why these should be more difficult to handle. For example, in \cite{EisenbudMR2810424}, vector bundles with natural cohomology are constructed for the polynomials $3xy- x- y-1,2 xy + y - 1,2xy + x-1,2 x y- 2$. These bundles are uniformly constructed as kernels of generic maps $\cO(-1,-1)^a\oplus \cO(-1,0)^b\xrightarrow{\phi}\cO(0,-1)^c \oplus \cO^d$.

However from the point of view $\bun_G$ there is a reason why the remaining cases are more difficult. To construct vector bundles $E$ with $\chi(E(x,y)) = r(x y - \gamma)$ one must consider bundles that have some jump lines. These come from maps in $\hom(\P^1, \bun_G)$ where the image hits more than one point. 

Substacks of $\bun_G$ with more than one point appear more complicated than $pt/Aut(F)$ which is one of the difficulties in extending the arguments here to the remaining cases. Nevertheless it seems plausible that an approach not too different than the one given here could be made to work for the remaining cases.

\bibliographystyle{plain} 

\bibliography{HuntingBundles}

\end{document}